\RequirePackage{fix-cm}
\documentclass[smallextended]{svjour3}       % onecolumn (second format)
\smartqed  % flush right qed marks, e.g. at end of proof
\usepackage{graphicx}
\usepackage{amsmath,amsfonts,amssymb,mathtools} % AMS
\usepackage{algorithm}
\usepackage[noend]{algorithmic}
\usepackage{cite}
\usepackage{xcolor}
\usepackage{hyperref}
\newcommand{\e}{\varepsilon}
\newcommand{\R}{{\mathbb R}}

\newcommand{\la}{\langle}
\newcommand{\ra}{\rangle}

\newtheorem{assumption}{Assumption}
\begin{document}

\title{
Tensor methods for strongly convex strongly concave saddle point problems and strongly monotone variational inequalities
\thanks{The work of P.~Ostroukhov was fulfilled in Sirius (Sochi) in August 2020 and was supported by Andrei M. Raigorodskii Scholarship in Optimization.  The research  of  P.~Dvurechensky was partially supported by the Ministry of Science and Higher Education of the Russian Federation (Goszadaniye) 075-00337-20-03, project no. 0714-2020-0005. The research of A. Gasnikov was funded by Math+ AA4-2 Scholarship in Optimization.}
}
% \subtitle{Do you have a subtitle?\\ If so, write it here}

\titlerunning{Tensor methods for strongly convex strongly concave SPP and strongly MVI}        % if too long for running head

\author{Petr Ostroukhov $^1$ \and
        Rinat Kamalov $^2$  \and
        Pavel Dvurechensky $^3$ \and
        Alexander Gasnikov $^4$
}


\institute{
$^1$ Moscow Institute of Physics and Technology, Dolgoprudny, Russia;
Institute for Information Transmission Problems RAS, Moscow, Russia
\email{ostroukhov@phystech.edu}
\and
$^2$ Moscow Institute of Physics and Technology, Dolgoprudny, Russia; V. A. Trapeznikov Institute of Control Sciences of Russian Academy of Sciences, Moscow, Russia
\email{kamalov.ra@phystech.edu}
\and
$^3$ Weierstrass Institute for Applied Analysis and Stochastics, Berlin; Institute for Information Transmission Problems RAS, Moscow \email{pavel.dvurechensky@wias-berlin.de}
\and
$^4$ Moscow Institute of Physics and Technology, Dolgoprudny, Russia; Weierstrass Institute for Applied Analysis and Stochastics, Berlin; Institute for Information Transmission Problems RAS, Moscow, Russia
\email{gasnikov@yandex.ru}
}

\date{Received: date / Accepted: date}

\maketitle

\begin{abstract}
    In this paper we propose three $p$-th order tensor methods for $\mu$-strongly-convex-strongly-concave saddle point problems (SPP). The first method is based on the assumption of $p$-th order smoothness of the objective and it achieves a convergence rate of $O \left( \left( \frac{L_p R^{p - 1}}{\mu} \right)^\frac{2}{p + 1} \log \frac{\mu R^2}{\e_G} \right)$, where $R$ is an estimate of the initial distance to the solution, and $\e_G$ is the error in terms of duality gap. Under additional assumptions of first and second order smoothness of the objective we connect the first method with a locally superlinear converging algorithm and develop a second method with the complexity of $O \left( \left( \frac{L_p R^{p - 1}}{\mu} \right)^\frac{2}{p + 1}\log \frac{L_2 R \max \left\{ 1, \frac{L_1}{\mu} \right\}}{\mu} + \log \frac{\log \frac{L_1^3}{2 \mu^2 \e_G}}{\log \frac{L_1 L_2}{\mu^2}} \right)$. The third method is a modified version of the second method, and it solves gradient norm minimization SPP with $\tilde O \left( \left( \frac{L_p R^p}{\e_\nabla} \right)^\frac{2}{p + 1} \right)$ oracle calls, where $\e_\nabla$ is an error in terms of norm of the gradient of the objective.
    Since we treat SPP as a particular case of variational inequalities, we also propose three methods for strongly monotone variational inequalities with the same complexity as the described above.

\end{abstract}

\section{Introduction}

    In this work we focus on two types of saddle point problems (SPP). The first one is the classic minimax problem:
    \begin{equation}\label{eq:saddle problem}
        \min_{x \in \mathcal X} \max_{y \in \mathcal Y} g(x, y),
    \end{equation}
    where $g: \mathcal X \times \mathcal Y \to \R$ is a convex over $\mathcal X$ and concave over $\mathcal Y$, and the sets  $\mathcal X, \mathcal Y$ are convex. This is a particular case of a more general problem, called monotone variational inequality (MVI). In MVI we have a monotone operator $F: \mathcal Z \to \R^n$ over a convex set $\mathcal Z \subset \R^n$ and we need to find
    \begin{equation}\label{eq:mvi}
        z^* \in \mathcal Z:\ \forall z \in \mathcal Z, \la F(z), z^* - z\ra \le 0.
    \end{equation}
    If we set $\mathcal Z = \mathcal X \times \mathcal Y$ and $ F(z) = (\nabla_x g(x,y),-\nabla_y g(x,y))$, then MVI is equivalent to the min-max SPP \eqref{eq:saddle problem}.
    
    The second problem is gradient norm minimization of SPP:
    \begin{equation}\label{eq:saddle problem gradient norm}
        \min_{(x, y) \in \mathcal X \times \mathcal Y} \|\nabla g(x, y)\|_2.
    \end{equation}
    
    For both problems we consider unconstrained case with $\mathcal X = \R^n$ and $\mathcal Y = \R^m$. Additionally, we assume $g(x, y)$ is $\mu$-strongly convex in $x \in \R^n$ and $\mu$-strongly concave in $y \in \R^m$.
    
    There is a number of papers on numerical methods for SPP \eqref{eq:saddle problem} in convex-concave setting \cite{korpelevich1976extragradient, tseng1995linear, nemirovski2004prox, nesterov2007dual, tseng2008accelerated}. One of the most popular among first-order methods for this setting is the Mirror-Prox algorithm \cite{nemirovski2004prox}, which treats saddle-point problems via solving the corresponding MVI. According to \cite{nemirovsky1983problem}, this method achieves optimal complexity of $O(1/\e)$ iterations for first-order methods applied to smooth convex-concave SPP in large dimensions. 
    
    Additional assumption of strong convexity and strong concavity lead to better results. The algorithms from  \cite{rockafellar1976monotone, tseng1995linear, nesterov2006solving, gidel2018variational, mokhtari2020unified} achieve iteration complexity of $O(L/\mu \log (1 / \e))$. In \cite{lin2020near} the authors proposed an algorithm with complexity $O(L/\sqrt{\mu_x \mu_y} \log^3(1/\e))$, which matches up to a logarithmic factor the lower bound, obtained in \cite{zhang2019lower}. It worths to mention that $\log^3(1/\e)$ factor can be improved, namely, it is possible to achieve iteration complexity of $O(L/\sqrt{\mu_x \mu_y} \log(1/\e))$ (see \cite{dvinskikh2020accelerated}).
    
    The methods listed above use first-order oracles, and it is known from optimization that tensor methods, which use higher-order derivatives, have faster convergence rate, yet for the price of more expensive iteration.
    The idea of using derivatives of high order in optimization is not new (see \cite{hoffmann1978higher-order}). The most common type of high-order methods use second-order oracles, for example Newton method \cite{nocedal2006numerical,nesterov1994interior} and its modifications such as the cubic regularized Newton method \cite{nesterov2006cubic}. Recently the idea of exploiting oracles beyond the second order started to attract increased attention, especially in convex optimization \cite{bullins2018fast, bullins2019higher, gasnikov2019optimal, pmlr-v99-gasnikov19b, dvurechensky2019near}.
    
    However, much less is known on high-order methods for SPP and MVIs. In \cite{monteiro2012iteration} the authors propose a second-order method based on their Hybrid Proximal Extragradient framework \cite{monteiro2010complexity}. The resulting complexity is $O(1/\e^\frac{2}{3})$.
    %Until this year it was unknown whether we can get better results in this area. 
    A recent work \cite{bullins2020higher} shows how to modify Mirror-Prox method using oracles beyond second order and improves complexity to reach duality gap $\e$ to $O(1/\e^\frac{2}{p + 1})$ for convex-concave problems with $p$-th order Lipschitz derivatives. The paper \cite{huang2020cubic} proposes a cubic regularized Newton method for solving SPP, which has global linear and local superlinear convergence rate if $\nabla g(x, y)$ and $\nabla^2 g(x, y)$ are Lipschitz-continuous and $g(x, y)$ is strongly convex in $x$ and strongly concave in $y$.
    
    In our work we make a next step and propose a Tensor method for strongly monotone variational inequalities and, as a corollary, a Tensor method for saddle point problems with strongly-convex-strongly-concave objective. 
    Standing on the ideas from \cite{bullins2020higher} and \cite{huang2020cubic}, our work can be split into three parts. 
    
    Firstly, we apply restart technique \cite{stonyakin2018generalized} to the HighOrderMirrorProx Algorithm \ref{alg:highordermirrorprox} from \cite{bullins2020higher}, which is possible because of strong convexity and strong concavity of the objective. Such a modification improves the algorithm complexity to $O \left( \left( \frac{L_p R^{p - 1}}{\mu} \right)^\frac{2}{p + 1} \log \frac{\mu R^2}{\e_G} \right)$, where $R$ is an upper bound for the initial distance to the solution $\|(x_1, y_1) - (x^*, y^*)\|_2$ and $L_p$ is the Lipschitz constant of the $p$-th derivative, and $\e_G$ is the error in terms of duality gap. 
    
    Secondly, using an estimate of the area of local superlinear convergence, when the algorithm reaches this area, we switch to the Cubic-Regularized Newton Algorithm \ref{alg:crn-spp} from \cite{huang2020cubic} to obtain local superlinear convergence of our algorithm. 
    The total complexity of the final Algorithm \ref{alg:restarted highordermirrorprox with crn-spp} becomes \newline $O \left( \left( \frac{L_p R^{p - 1}}{\mu} \right)^\frac{2}{p + 1}\log \frac{L_2 R \max \left\{ 1, \frac{L_1}{\mu} \right\}}{\mu} + \log \frac{\log \frac{L_1^3}{2 \mu^2 \e_G}}{\log \frac{L_1 L_2}{\mu^2}} \right)$, where $L_1$ and $L_2$ are Lipschitz constans for first and second order derivatives respectively. We want to emphasize, that the obtained $\log \log (1/\e)$ dependency on $\e$ cannot be improved even in convex optimization \cite{kornowski2020high}.

    Thirdly, we apply framework from \cite{dvurechensky2019near} to the Algorithm \ref{alg:restarted highordermirrorprox with crn-spp} to solve the problem \eqref{eq:saddle problem gradient norm} and obtain the Algorithm \ref{alg:restarted highordermirrorprox with crn-spp gradient norm}. Its convergence rate is $\tilde O \left( \left( \frac{L_p R^p}{\e_\nabla} \right)^\frac{2}{p + 1} \right)$, where by tilde we mean additional multiplicative $\log$ factor, and $\e_\nabla$ is an error in terms of gradient norm of the objective.
    
    Our paper is organized as follows. First of all, in Section \ref{sec:preliminaries} we provide necessary notations and assumptions (Section \ref{sec:assumptions}). Then, we present the new algorithm and obtain its convergence rate in Section \ref{sec:main results}. Firstly, in Section \ref{sec:restarted highordermirrorprox} we talk only about restarted algorithm from \cite{bullins2020higher} and get its complexity. Secondly, in Section \ref{sec:local quadratic convergence} we describe how to connect it to Algorithm \ref{alg:crn-spp} from \cite{huang2020cubic} in its quadratic convergence area and get the final Algorithm \ref{alg:restarted highordermirrorprox with crn-spp} convergence rate. Thirdly, in Section \ref{sec:gradient norm} we focus on how to wrap Algorithm \ref{alg:restarted highordermirrorprox with crn-spp} in a framework from \cite{dvurechensky2019near} and obtain its complexity. Finally, in Section \ref{sec:discussion} we discuss our results and present some possible directions for future work. 
    
\section{Preliminaries}\label{sec:preliminaries}

    We use $z \in \R^n \times \R^m$ to denote the pair $(x,y)$, $\nabla^p g(z)[h_1, ..., h_p],\ p \ge 1$ to denote directional derivative of $g$ at $z$ along directions $h_i \in \R^n \times \R^m, i = 1, ..., p$. The norm of the $p$-th order derivative is defined as
    \[
        \|\nabla^p g(z)\|_2 := \max_{h_1,...,h_p \in \R^n \times \R^m} \{ |\nabla^p g(z)[h_1,...,h_p]| : \|h_i\|_2 \leq 1, i = 1,...,p\}
    \]
    or equivalently
    \[
        \|\nabla^p g(z)\|_2 := \max_{h \in \R^n \times \R^m} \{ |\nabla^p g(z)[h]^p| : \|h\|_2 \leq 1\}.
    \]
    Here we denote $\nabla^p g(z)[h,...,h]$ as $\nabla^p g(z)[h]^p$. Also here and below $\| \cdot \|_2$ is a Euclidean norm for vectors. 
    
    Taylor approximation of some function $f$ at point $z$ up to the order of $p$ we denote by
    \[
        \Phi^f_{z, p}(\hat z) := \sum_{i = 0}^p \frac{1}{i!} \nabla^i f(z) [\hat z - z]^i.
    \]
    For ease of notation, the Taylor approximation of the objective $g$ we denote by $\Phi_{(x, y), p}(\hat x, \hat y) \equiv \Phi_{z, p}(\hat z) \equiv \Phi^g_{z, p}(\hat z)$.
    
    By $D: \mathcal Z \times \mathcal Z \to \R^n$ we denote Bregman divergence induced by a function $d: \mathcal Z \to \R$, which is continuously-differentiable and $1$-strongly convex.
    The definition of Bregman divergence is
    \[
        D(z_1, z_2) := d(z_1) - d(z_2) - \la \nabla d(z_2), z_1 - z_2 \ra.
    \]
    In our paper we use half of squared Euclidean distance as Bregman divergence 
    % \pd{usually it is with $\frac{1}{2}$ multiplier}
    \begin{equation}\label{eq:bregman div, euclidean distance}
        D(z_1, z_2) = \frac{1}{2}\|z_1 - z_2\|_2^2.
    \end{equation}
    
    During the analysis of convergence of our approach for gradient norm minimization \eqref{eq:saddle problem gradient norm} we will need the regularized Taylor approximation of objective $g$:
    \begin{gather*}
        \Omega_{(x, y), p, L_p}(\hat x, \hat y) := \\ 
        \Phi_{(x, y), p}(\hat x, \hat y) + \frac{L_p (\sqrt{2})^{p - 1}}{(p + 1)!} \|\hat x - x\|_2^{p + 1} - \frac{L_p (\sqrt{2})^{p - 1}}{(p + 1)!} \|\hat y - y\|_2^{p + 1}.
    \end{gather*}
    Its min-max point we denote by
    \[
        T_{p,L_p}^g \left( x, y \right)\in \mbox{Arg} \min_{\tilde x \in \R^n} \max_{\tilde y\in \R^m} \left\{ \Omega_{(x, y), p, L_p}(\tilde x, \tilde y) \right\}.
    \]
    
    As we mentioned earlier, in this paper we consider two types of SPP: classic minimax problem \eqref{eq:saddle problem} and gradient norm minimization \eqref{eq:saddle problem gradient norm}. We need to introduce the definitions of approximate solutions of these problems. We use different indices in error notations for these problems to avoid ambiguity. 
    
    Firstly, the problem \eqref{eq:saddle problem} is usually solved in terms of the duality gap
    \begin{equation}\label{eq:dual gap}
        G_{\mathcal X \times \mathcal Y}(x, y) := \max_{y' \in \mathcal Y} g(x, y') - \min_{x' \in \mathcal X} g(x', y).
    \end{equation}
    Since in our case $\mathcal X = \R^n$ and $\mathcal Y = \R^m$, we drop the notations of these sets from index of the duality gap and denote duality gap just as $G(x, y)$. Then, we define $\e_G$-approximate solution of \eqref{eq:saddle problem}:
    \begin{equation}\label{eq:saddle problem approximate}
        \tilde x^* \in \R^n, \tilde y^* \in \R^m \Rightarrow G(\tilde x^*, \tilde y^*) \le \e_G.
    \end{equation}
    
    Secondly, for the problem \eqref{eq:saddle problem gradient norm} we don't need any additional functionals, and $\e_\nabla$-approximate solution of \eqref{eq:saddle problem gradient norm} is of the form
    \begin{equation}\label{eq:saddle problem gradient norm approximate}
        \tilde x^* \in \R^n, \tilde y^* \in \R^m \Rightarrow \|\nabla g(\tilde x^*, \tilde y^*)\|_2 \le \e_\nabla.
    \end{equation}

    \subsection{Assumptions}\label{sec:assumptions}

        We assume objective $g$ is strongly convex, strongly concave and $p$-times differentiable.
        \begin{assumption}\label{ass:strong convexity concavity}
            $g(x, y)$ is $\mu$-strongly convex in $x$ and $\mu$-strongly concave in $y$. 
        \end{assumption}
        Recall that the definition of strong convexity and strong concavity is as follows.
        \begin{definition}
            $g: \R^n \times \R^m \to \R$ is called $\mu$-strongly convex and $\mu$-strongly concave if
            \begin{gather}
                \forall x_1, x_2 \in \R^n,\ y \in \R^m \Rightarrow \la \nabla_x g(x_1, y) - \nabla_x g(x_2, y), x_1 - x_2 \ra \ge \mu \|x_1 - x_2\|_2^2, \label{eq:strong convexity-concavity}  \\
                \forall y_1, y_2 \in \R^m,\ x \in \R^n \Rightarrow \la -\nabla_y g(x, y_1) + \nabla_y g(x, y_2), y_1 - y_2 \ra \ge \mu \|y_1 - y_2\|_2^2.
            \end{gather}
        \end{definition}
        
        Before showing the connection between problem \eqref{eq:saddle problem} and MVI \eqref{eq:mvi} we need the definition of strong monotonicity.
        \begin{definition}
            $F: \mathcal Z \to \R^n$ is strongly monotone if
            \begin{equation}\label{eq:strongly monotone}
                \la F(z_1) - F(z_2), z_1 - z_2 \ra \ge \mu\|z_1 - z_2\|_2^2.
            \end{equation}
        \end{definition}
        Denote $z = \begin{pmatrix} x \\ y \end{pmatrix}$, and operator $F: \R^n \times \R^m \to \R^n \times \R^m$:
        \begin{equation}\label{eq:g to F}
            F(z) = F(x, y) := \begin{pmatrix} \nabla_x g(x, y) \\ -\nabla_y g(x, y) \end{pmatrix}.
        \end{equation}
        % Since $g(x, y)$ is strongly-convex-strongly-concave, it is easy to show, that the operator $F(z)$ is strongly monotone. 
        %According to these considerations, when we solve the min-max problem \eqref{eq:saddle problem} we also tackle the MVI problem \eqref{eq:mvi} with specific operator $F$. 
        According to these definitions, the min-max problem \eqref{eq:saddle problem} can be tackled via solving the MVI problem \eqref{eq:mvi} with the specific operator $F$ given in \eqref{eq:g to F}. In our work we use the following assumptions. 
        
        \begin{assumption}\label{ass:1-lipschitz}
            $F(z)$ satisfies first order Lipschitz condition:
            \begin{gather}
                \|F(z_1) - F(z_2)\|_2 \le L_1 \|z_1 - z_2\|_2 \notag \\
                \Leftrightarrow \|\nabla g(z_1) - \nabla g(z_2)\|_2 \le L_1 \|z_1 - z_2\|_2. \label{eq:1-lipschitz}
            \end{gather}
        \end{assumption}
        
        \begin{assumption}\label{ass:2-lipschitz}
            $F(z)$ satisfies second order Lipschitz condition:
            \begin{gather}
                \|\nabla F(z_1) - \nabla F(z_2)\|_2 \le L_2 \|z_1 - z_2\|_2 \notag \\
                \Leftrightarrow \|\nabla^2 g(z_1) - \nabla^2 g(z_2)\|_2 \le L_2 \|z_1 - z_2\|_2. \label{eq:2-lipschitz}
            \end{gather}
        \end{assumption}
        
        \begin{assumption}\label{ass:p-lipschitz}
            $F(z)$ satisfies $p$-th order Lipschitz condition ($p$-smooth):
            \begin{gather}
                \|\nabla^{p-1} F(z_1) - \nabla^{p-1} F(z_2)\|_2 \le L_p \|z_1 - z_2\|_2 \notag \\
                 \Leftrightarrow \|\nabla^p g(z_1) - \nabla^p g(z_2)\|_2 \le L_2 \|z_1 - z_2\|_2. \label{eq:p-lipschitz}
            \end{gather}
        \end{assumption}
        We should note, that, to be consistent with \cite{bullins2020higher}, we define $p$-th order smoothness (Lipschitzness) of $F$ as a property of $(p - 1)$-th derivative of $F$, and, therefore, as a property of $p$-th derivative of $g$. 

\section{Main results}\label{sec:main results}
    Firstly, in this section we propose the algorithm for finding $\e_G$-approximate solution to problem \eqref{eq:saddle problem approximate}, where $ g \colon \mathbb{R}^n \times \mathbb{R}^m  \to \R$ is $p$-smooth and $\mu$-strongly-convex-concave (assumptions \ref{ass:p-lipschitz} and \ref{ass:strong convexity concavity}), which allows to achieve iteration complexity of $O \left( \left( \frac{L_p R^{p - 1}}{\mu} \right)^\frac{2}{p + 1} \log \frac{\mu R^2}{\e_G} \right)$, where $R \geqslant \| z_1 - z^{*} \|_2$. This algorithm is a restarted modification of Algorithm \ref{alg:highordermirrorprox}.
    
    Secondly, we develop the algorithm for tackling the same problem, where $g$ is first, second and $p$-th order Lipschitz and $ \mu $-strongly-convex-concave function (all assumptions \ref{ass:strong convexity concavity}, \ref{ass:1-lipschitz}, \ref{ass:2-lipschitz}, \ref{ass:p-lipschitz}). It involves the idea of exploiting previous algorithm and then switching to the Algorithm \ref{alg:crn-spp} in its quadratic convergence area. Thus, we obtain the Algorithm \ref{alg:restarted highordermirrorprox with crn-spp}, that allows to achieve iteration complexity of $O \left( \left( \frac{L_p R^{p - 1}}{\mu} \right)^\frac{2}{p + 1}\log \frac{L_2 R \max \left\{ 1, \frac{L_1}{\mu} \right\}}{\mu} + \log \frac{\log \frac{L_1^3}{2 \mu^2 \e_G}}{\log \frac{L_1 L_2}{\mu^2}} \right)$.
    
    Thirdly, we propose the algorithm to find $\e_\nabla$-approximate solution to problem \eqref{eq:saddle problem gradient norm approximate}, where all the assumptions \ref{ass:strong convexity concavity}, \ref{ass:1-lipschitz}, \ref{ass:2-lipschitz}, \ref{ass:p-lipschitz} hold. To achieve this we use the Algorithm \ref{alg:restarted highordermirrorprox with crn-spp}, which we mentioned earliear, inside the framework from \cite{dvurechensky2019near}. Final complexity of such algorithm in terms of norm of the gradient is $\tilde O \left( \left( \frac{L_p R^p}{\e_\nabla} \right)^\frac{2}{p + 1} \right)$, where by tilde we mean additional multiplicative $\log$ factor.
    
    \subsection{Restarted HighOrderMirrorProx}\label{sec:restarted highordermirrorprox}
    
        % As a next step we introduce Algorithm \ref{alg:crn-spp} and theorem, that proves its local quadratic convergence. This theorem is based on Theorem 3.6 from \cite{huang2020cubic}.
    
        As mentioned earlier, in this subsection we provide restarted modification of Algorithm \ref{alg:highordermirrorprox}. But, initially, we need to give some additional information from \cite{bullins2020higher}.
        
        Since our goal is an approximate solution to MVI, we define its $\e$-approximate solution as
        \begin{equation}\label{eq:eps approximate solution}
            z^* \in \mathcal Z: \forall z \in \mathcal Z \Rightarrow \la F(z), z^* - z \ra \le \e.
        \end{equation}
        At the same time, the bounds of Algorithm \ref{alg:highordermirrorprox} is of the form
        \begin{equation}\label{eq:bullins eps approximate solution}
            \forall z \in \mathcal Z \Rightarrow \frac{1}{\Gamma_T} \sum_{t = 1}^T \gamma_t \la F(z_t), z_t - z \ra \le \e,
        \end{equation}
        where points $z_t$ and $\gamma_t > 0$ are produced by the Algorithm \ref{alg:highordermirrorprox}, and $\Gamma_T = \sum_{t = 1}^T \gamma_t$. The following lemma establishes the relation between \eqref{eq:eps approximate solution} and \eqref{eq:bullins eps approximate solution}.
        
        \begin{lemma}[Lemma 2.7 from \cite{bullins2020higher}]\label{lemma:mvi eps solutions connection}
            Let $F: \mathcal Z \to \R^n,$ be monotone, $z_t \in \mathcal Z,\ t = 1, ..., T$, and let $\gamma_t > 0$. Let $\bar z_t = \frac{1}{\Gamma_T} \sum_{t = 1}^T \gamma_t z_t$. Assume \eqref{eq:bullins eps approximate solution} holds. Then $\bar z_t$ is an $\e$-approximate solution to \eqref{eq:mvi}.
        \end{lemma}
        
        MVI problem \eqref{eq:mvi}, which is sometimes called "weak MVI", is closely connected to strong MVI problem, where we need to find 
        \begin{equation}\label{eq:strong mvi}
            z^* \in \mathcal Z: \forall z \in \mathcal Z \Rightarrow \la F(z^*), z^* - z \ra \le 0.
        \end{equation}
        If $F$ is continuous and monotone, the problems \eqref{eq:mvi} and \eqref{eq:strong mvi} are eqiuvalent.
    
        The convergence rate of the Algorithm \ref{alg:highordermirrorprox} is stated in the following lemma.
        
        \begin{algorithm}[t]
        	\caption{HighOrderMirrorProx [Algorithm 1 in \cite{bullins2020higher}]}\label{alg:highordermirrorprox}
        	\begin{algorithmic}[1]
        		\STATE \textbf{Input} $z_1 \in \mathcal Z, p \ge 1, T > 0$.
        		
        		\FOR{t = 1 \TO T}
        		\STATE Determine $\gamma_t$, $\hat z_t$ such that:
        		\begin{gather*}
        		    \hat z_t = \arg \min_{z \in \mathcal Z} \{\gamma_t \langle \Phi^F_{z_t, p}(\hat z_t), z - z_t \rangle + D(z, z_t) \}, \\
        		    \frac{p!}{32 L_p \|\hat z_t - z_t\|^{p - 1}_2} \le \gamma_t \le \frac{p!}{16 L_p \|\hat z_t - z_t\|_2^{p - 1}}, \\
        		    z_{t + 1} = \arg \min_{z \in \mathcal Z} \{\langle \gamma_t F(\hat z_t), z - \hat z_t \rangle + D(z, z_t)\}.
        		\end{gather*}
        		\ENDFOR
        		\STATE Define $\Gamma_T \overset{def}{=} \sum_{t = 1}^T \gamma_t$
        		\RETURN $\bar z_T \overset{def}{=} \frac{1}{\Gamma_T} \sum_{t = 1}^T \gamma_t \hat z_t$.
        	\end{algorithmic}
        \end{algorithm}
        
        \begin{lemma}[Lemma 4.1 from \cite{bullins2020higher}]
            Suppose $F: \mathcal Z \to \mathbb R^n$ is $p$-th order Lipschitz and let $\Gamma_T = \sum_{t = 1}^T \gamma_t$. Then, the iterates $\{\hat z_t\}_{t \in [T]}$, generated by Algorithm 1, satisfy
            
            \begin{equation}\label{eq:highordermirrorprox complexity}
                \forall z \in \mathcal Z \Rightarrow \frac{1}{\Gamma_T} \sum_{t = 1}^T \langle \gamma_t F(\hat z_t), \hat z_t - z \rangle \le \frac{16 L_p}{p!} \bigg(\frac{D(z, z_1)}{T}\bigg)^\frac{p + 1}{2}.
            \end{equation}
        \end{lemma}
        Thus, these two lemmas tell us, that if $z_t$ and $\gamma_t$ are generated by the Algorithm \ref{alg:highordermirrorprox}, and the right hand side of \eqref{eq:highordermirrorprox complexity} is smaller than $\e$, then $\bar z_t = \frac{1}{\Gamma_T} \sum_{t = 1}^T \gamma_t z_t$ is an $\e$-solution to regular MVI \eqref{eq:eps approximate solution}. Hence, it is also a solution to a convex-concave SPP. The natural way to improve the method for convex-concave problem in tighter strongly-convex-strongly-concave setting is to use restarts \cite{stonyakin2018generalized}. As a result, we obtain Algorithm \ref{alg:restarted highordermirrorprox}. 
        
        \begin{algorithm}[t]
        	\caption{Restarted HighOrderMirrorProx}\label{alg:restarted highordermirrorprox}
        	\begin{algorithmic}[1]
        		\STATE \textbf{Input} $z_1 \in \mathcal Z, p \ge 1, 0 < \e_G < 1$, $R: R \geq \|z_1 - z^*\|_2$.
        		
        		\STATE $k = 1$
        		\STATE $\tilde z_1 = z_1$
                \FOR{$i \in [n]$, where $n = \left\lceil \frac{1}{2} \log \frac{\mu R^2}{\e_G} \right\rceil$}
                    \STATE Set $R_i = \frac{R}{2^{i - 1}}$
                
        		    \STATE Set $T_i = \left\lceil \bigg( \frac{64 L_p R_i^{p - 1}}{\mu} \bigg)^\frac{2}{p + 1} \right\rceil$
        		    
        		    \STATE Run Algorithm \ref{alg:highordermirrorprox} with $\tilde z_{i}$, $p$,  $T_i$ as input
        		    
        		    \STATE $\tilde z_{i + 1} = \bar z_{T_i}$
            	\ENDFOR
            	\RETURN $\tilde z_{i}$
        	\end{algorithmic}
        \end{algorithm}
        
        \begin{theorem}\label{th:restarted highordermirrorprox conv}
            Suppose $F: \R^n \times \R^m \to \R^n \times \R^m$, that is defined in \eqref{eq:g to F}, is $p$-th order Lipschitz and $\mu$-strongly monotone (Assumptions \ref{ass:strong convexity concavity} and \ref{ass:p-lipschitz} hold).  Denote $R$ such that $ R \ge \|z_1 - z^*\|_2$. Then Algorithm \ref{alg:restarted highordermirrorprox} complexity is
            \begin{equation}
                O \left( \left( \frac{L_p R^{p - 1}}{\mu} \right)^\frac{2}{p + 1} \log \frac{\mu R^2}{\e_G} \right).
            \end{equation}
        \end{theorem}
        \begin{proof}
            From \eqref{eq:strong mvi} and \eqref{eq:highordermirrorprox complexity} we get the following:
            \begin{equation}\label{eq:th_1 eq_1}
                \sum_{t = 1}^T \gamma_t \langle F(\hat z_t) - F(z^*); \hat z_t - z^* \rangle
                    \le \frac{16 L_p}{p!} \bigg(\frac{\| z_1 - z^*\|_2^2}{2T}\bigg)^\frac{p + 1}{2}.
            \end{equation}
            From this and the fact that $F(x)$ is $\mu$-strongly monotone we have
            \begin{align}
                \mu \|\bar z_T - z^*\|_2^2 
                    \overset{(*)}{\le} \frac{\mu}{\Gamma_T} \sum_{t = 1}^T \gamma_t \|\hat z_t - z^*\|_2^2 
                    &\overset{\eqref{eq:strongly monotone}}{\le} \frac{1}{\Gamma_T} \sum_{t = 1}^T \gamma_t \langle F(\hat z_t) - F(z^*); \hat z_t - z^* \rangle \label{eq:th_1 eq_2} \\
                    &\overset{\eqref{eq:th_1 eq_1}}{\le} \frac{16 L_p}{p!} \bigg(\frac{\|z_1 - z^*\|_2^2}{2T}\bigg)^\frac{p + 1}{2}, \notag
            \end{align}
            where (*) follows from convexity of $\|z\|_2^2$.
            
            Now we restart the method every time the distance to solution decreases at least twice. Let $T_i$ be such that $\|\bar z_{T_i} - z^*\|_2 \le \frac{\|\tilde z_i - z^*\|_2}{2}$, where $\tilde z_i$ is the point, where we restart our algorithm. Denote $R_1 = R \ge \|\tilde z_1 - z^*\|_2$, $R_i = R_1/2^{i - 1} \ge \|\tilde z_i - z^*\|_2$. Then the number of iterations before $(i + 1)$-th restart is
            \begin{gather*}
                \mu \|\bar z_{T_i} - z^*\|_2^2 \overset{\eqref{eq:th_1 eq_2}}{\le} \frac{16 L_p}{p!} \left( \frac{\|\tilde z_i - z^*\|_2^2}{2T_i} \right)^\frac{p + 1}{2} \le  \frac{16 L_p}{p!} \bigg( \frac{R_i^2}{2T_i} \bigg)^\frac{p + 1}{2} \le \frac{\mu \| \tilde z_i - z^*\|_2^2}{4} \le \frac{\mu R_i^2}{4}\\
                \Leftrightarrow T_i \ge \frac{R_i^2}{2} \bigg( \frac{64 L_p}{p! \mu R_i^2} \bigg)^\frac{2}{p + 1} \ge \bigg( \frac{64 L_p R_i^{p - 1}}{\mu} \bigg)^\frac{2}{p + 1} = \bigg\lceil \bigg( \frac{64 L_p R_i^{p - 1}}{\mu} \bigg)^\frac{2}{p + 1} \bigg\rceil.
            \end{gather*}
            Next we need to obtain the number of restarts $n$, required to achieve the desired accuracy. From \eqref{eq:th_1 eq_1} we get
            \begin{align*}
                \frac{1}{\Gamma_{T_n}} \sum_{t = 1}^{T_n} \gamma_t \langle F(\hat z_t) - F(z^*); \hat z_t - z^* \rangle &\le \frac{16 L_p}{p!} \bigg(\frac{\|\tilde z_n - z^*\|_2^2}{2T_n}\bigg)^\frac{p + 1}{2} \\
                &\le 16 L_p \left(\frac{R_n^2}{\left( \frac{64 L_p R_n^{p - 1}}{\mu} \right)^\frac{2}{p + 1}} \right)^\frac{p + 1}{2} \\
                &= \frac{\mu R_n^2}{4} = \frac{\mu R^2}{2^{2n}} \le \e_G.
            \end{align*}
            \[
                \Leftrightarrow n \ge \frac{1}{2} \log \frac{\mu R^2}{\e_G} = \left\lceil \frac{1}{2} \log \frac{\mu R^2}{\e_G} \right\rceil.
            \]
            
            Finally, the total number of iterations is
            \begin{align*}
                N &= \sum_{i = 1}^n T_i = \sum_{i = 1}^n \bigg\lceil \bigg( \frac{64 L_p R_i^{p - 1}}{\mu} \bigg)^\frac{2}{p + 1} \bigg\rceil \le \left( \frac{64 L_p}{\mu} \right)^\frac{2}{p + 1} \sum_{i = 1}^n R_i^\frac{2(p - 1)}{p + 1} + n \\ 
                &\le \left( \frac{64 L_p R^{p - 1}}{\mu} \right)^\frac{2}{p + 1} n + n \\
                &= \left( \frac{64 L_p R^{p - 1}}{\mu} \right)^\frac{2}{p + 1} \left\lceil \frac{1}{2} \log \frac{\mu R^2}{\e_G} \right\rceil + \left\lceil \frac{1}{2} \log \frac{\mu R^2}{\e_G} \right\rceil \\
                &= O \left( \left( \frac{L_p R^{p - 1}}{\mu} \right)^\frac{2}{p + 1} \log \frac{\mu R^2}{\e_G} \right).
            \end{align*}
            This completes the proof.
            \qed
        \end{proof}
    
    \subsection{Local quadratic convergence}\label{sec:local quadratic convergence}
    
        Just like in previous subsection, becides introducing the Algorithm \ref{alg:crn-spp} and its convergence rate we need to provide some prerequisite information from \cite{huang2020cubic}.
    
        \begin{algorithm}[t]
        	\caption{CRN-SPP [Algorithm 1 in \cite{huang2020cubic}]}\label{alg:crn-spp}
        	\begin{algorithmic}[1]
        	    \STATE \textbf{Input} $z_0$, $\e$, $\bar \gamma > 0$, $\rho, \alpha \in (0, 1)$, $g$ satisfies Assumptions \ref{ass:strong convexity concavity}, \ref{ass:1-lipschitz} and \ref{ass:2-lipschitz}.
        	    \WHILE{$m(z_k) > \e$}
        	        \STATE $\gamma_k = \bar \gamma$
        	        \WHILE{True}
        	            \STATE Solve the subproblem $(\tilde x_{k + 1}, \tilde y_{k + 1}) = \arg \min_x \max_y g_k(x, y; \gamma_k)$
        	           % \STATE $u_k = ,\ v_k = $.
        	            \IF{$\gamma_k (\|\tilde x_{k + 1} - x_k\| + \|\tilde y_{k + 1} - y_k\|) > \mu$}
        	                \STATE $\gamma_{k} = \rho \gamma_k$
        	            \ELSE 
        	                \STATE \textbf{break}
        	            \ENDIF
        	        \ENDWHILE
        	        \STATE $d_k = (\tilde x_{k + 1} - x_k; \tilde y_{k + 1} - y_k)$
        	        \IF{$m(z_k + \alpha d_k) < m(z_k + d_k)$}
        	            \STATE $z_{k + 1} = z_k + \alpha d_k$
        	        \ELSIF{$m(z_k + \alpha d_k) \ge m(z_k + d_k)$}
                        \STATE $z_{k + 1} = z_k + d_k$
                    \ENDIF
                    \STATE $k = k + 1$
        	    \ENDWHILE
        	    \RETURN $z_k$
        	\end{algorithmic}
        \end{algorithm}
        
        Because of strong convexity and strong concavity of $g(x, y)$ a unique solution $z^*$ to a SPP \eqref{eq:saddle problem} exists, and $F(z^*) = 0$. Thus, we can use the following merit function from \cite{huang2020cubic} during analysis of Algorithm \ref{alg:crn-spp} complexity.
        \begin{equation}\label{eq:merit function}
            m(z) := \frac{1}{2}\|F(z)\|_2^2 = \frac{1}{2} (\|\nabla_x g(x, y)\|_2^2 + \|\nabla_y g(x, y)\|^2_2).
        \end{equation}
        Algorithm \ref{alg:crn-spp} solves additional saddle point subproblem on each step, that we denote as
        \begin{gather*}
            \min_{x \in \R^n} \max_{y \in \R^m} g_k(x, y, \gamma_k) := \\ g(z_k) + \la \nabla g(z_k), z - z_k \ra + \frac{1}{2} \nabla^2 g(z_k)[z - z_k]^2 + \frac{\gamma_k}{3} \|x - x_k\|_2^3 - \frac{\gamma_k}{3} \|y - y_k\|_2^3,
        \end{gather*}
        where $\gamma_k$ is some constant.
        
        This proposition provides the relation between the merit function $m(z)$ and the duality gap under assumptions \ref{ass:strong convexity concavity} and \ref{ass:1-lipschitz}.
        \begin{proposition}[Proposition 2.5 from \cite{huang2020cubic}]\label{prop:merit function and dual gap}
            Let assumptions \ref{ass:strong convexity concavity} and \ref{ass:1-lipschitz} hold. For problem \eqref{eq:saddle problem} and any point $z = (x, y)$ the duality gap \eqref{eq:dual gap} and the merit function \eqref{eq:merit function} satisfy the following inequalities
            \begin{equation}\label{eq:merit function and dual gap}
                \frac{\mu}{L_1^2} m(z) \le G(x, y) \le \frac{L_1}{\mu^2} m(z).
            \end{equation}
        \end{proposition}
        
        The next theorem proves local quadratic convergence of the Algorithm \ref{alg:crn-spp}, and it is based on Theorem 3.6 from \cite{huang2020cubic}.
        \begin{theorem}[Theorem 3.6 from \cite{huang2020cubic}]\label{th:crn-spp local quadratic conv}
            Suppose $F: \mathcal Z \to \mathbb R^n$ is $\mu$-strongly monotone, first and second order Lipschitz operator (assumptions \ref{ass:strong convexity concavity}, \ref{ass:1-lipschitz} and \ref{ass:2-lipschitz} hold).
            Let $\{z_k\}$ be generated by Algorithm \ref{alg:crn-spp} with $\bar \gamma = \frac{L_2 \mu^2}{2 L^2}$, $\xi = \max \left\{ 1, \frac{L_1}{\mu} \right\}$ and
            \begin{equation}\label{eq:crn-spp quadratic conv area}
                z_0: \|z_0 - z^*\|_2 \le \frac{\mu}{L_2 \xi}.
            \end{equation}
            Then
            \begin{equation}\label{eq:crn-spp quadratic conv}
                \forall k \ge 0\ \|z_{k + 1} - z^*\|_2 \le \frac{L_2 \xi}{\mu} \|z_k - z^*\|_2^2, 
            \end{equation}
        \end{theorem}
        \begin{proof}
            Here we provide only the modified part of its proof. The rest of it can be found in \cite{huang2020cubic}.
            
            If $z_{k + 1} = \tilde z_{k + 1} = z_k + d_k$, then
            \[
                \|z_{k + 1} - z^*\|_2 = \|\tilde z_{k + 1} - z^*\|_2 \le \frac{L_2}{\mu} \|z^k - z^*\|^2_2 \le \frac{L_2 \xi}{\mu} \|z_k - z^*\|_2^2.
            \]
            Else if $z_{k + 1} = \hat z_{k + 1} = z_k + \alpha d_k$, then
            \[
                \|z_{k + 1} - z^*\|_2 = \|\hat z_{k + 1} - z^*\|_2 \le \frac{L_1 L_2}{\mu^2} \|z^k - z^*\|_2^2 \le \frac{L_2 \xi}{\mu} \|z^k - z^*\|_2^2.
            \]
            Hence, we get \eqref{eq:crn-spp quadratic conv}.
            
            Now we need to find the area, where \eqref{eq:crn-spp quadratic conv} works:
            \begin{gather*}
                \exists c: \forall k \ge 0: \|z_k - z^*\|_2 \le c \Rightarrow \|z_{k + 1} - z^*\|_2 \le \frac{L_2 \xi}{\mu} \|z_k - z^*\|_2^2 \\
                \Leftrightarrow \|z_{k + 1} - z^*\|_2 \le \frac{L_2 \xi}{\mu} \|z_k - z^*\|_2 \le \frac{L_2 \xi c^2}{\mu} = c \\
                \Leftrightarrow c = \frac{\mu}{L_2 \xi}.
            \end{gather*}
            Thus, we get \eqref{eq:crn-spp quadratic conv area}.
        \qed
        \end{proof}
        Our idea is to use Algorithm \ref{alg:restarted highordermirrorprox} until it reaches the area \eqref{eq:crn-spp quadratic conv area} and then switch to Algorithm \ref{alg:crn-spp}. Algorithm \ref{alg:restarted highordermirrorprox with crn-spp} provides the pseudocode of this idea. From Proposition \ref{prop:merit function and dual gap}, our Theorem \ref{th:restarted highordermirrorprox conv} and Theorem \ref{th:crn-spp local quadratic conv}, we obtain the complexity of Algorithm \ref{alg:restarted highordermirrorprox with crn-spp}.
        
        \begin{algorithm}[t]
        	\caption{Restarted HighOrderMirrorProx with local quadratic convergence}\label{alg:restarted highordermirrorprox with crn-spp}
        	\begin{algorithmic}[1]
        		\STATE \textbf{Input} $z_1 \in \mathcal Z, p \ge 1, 0 < \e_G < 1, R: R\ge \|z_1 - z^*\|_2$, $\rho \in (0, 1)$, $\alpha \in (0, 1)$.
        		
        		\STATE $\tilde z_{1} = z_1$
        		
                \FOR{$i \in [n]$, where $n = \left\lceil \log \frac{L_2 R \xi}{\mu} + 1 \right\rceil$}
                    \STATE Set $R_i = \frac{R}{2^{i - 1}}$
                
        		    \STATE Set $T_i = \bigg\lfloor \frac{R_i^2}{2} \bigg( \frac{64L_p}{p! \mu R_i} \bigg)^\frac{2}{p + 1} \bigg\rfloor$
        		    
        		    \STATE Run Algorithm \ref{alg:highordermirrorprox} with $\tilde z_{i}$, $p$,  $T_i$ as input
        		    \STATE $\tilde z_{i + 1} = \bar z_{T_i}$
            	\ENDFOR
                	
            	\STATE Run Algorithm \ref{alg:crn-spp} with $\tilde z_{i + 1}$, $\tilde \e = \frac{\mu^2 \e_G}{L}$, $\bar \gamma = \frac{L_2 \mu^2}{2 L_1^2}$, $\rho$, $\alpha$, $g$ as input
            	\RETURN $z_k$
        	\end{algorithmic}
        \end{algorithm}
        
        \begin{theorem}\label{th:restarted highordermirrorprox with local quadratic conv}
            Suppose $F: \mathbb R^n \times \mathbb{R}^m \to \mathbb R^n \times \mathbb{R}^m $, that is defined in \eqref{eq:g to F}, is $\mu$-strongly monotone, first, second and $p$-th order Lipschitz operator (all assumptions \ref{ass:strong convexity concavity}, \ref{ass:1-lipschitz}, \ref{ass:2-lipschitz}, \ref{ass:p-lipschitz} hold). Denote $R: R \geq \|z_1 - z^*\|_2$ and $\xi = \max \left\{ 1, \frac{L_1}{\mu} \right\}$. Then the complexity of Algorithm \ref{alg:restarted highordermirrorprox with crn-spp}  is
            \begin{equation}
                O \left( \left( \frac{L_p R^{p - 1}}{\mu} \right)^\frac{2}{p + 1} \log \frac{L_2 \xi R}{\mu} + \log \frac{\log \frac{L_1^3}{2 \mu^2 \e_G}}{\log \frac{L_1 L_2}{\mu^2}} \right).
            \end{equation}
        \end{theorem}
        
        \begin{proof}
            First of all, we need to find the number of restarts $n$ of Algorithm \ref{alg:restarted highordermirrorprox} to reach the area of local quadratic convergence of Algorithm \ref{alg:crn-spp} from \eqref{eq:crn-spp quadratic conv area}: $\|\tilde z_n - z^*\|_2 \le \frac{\mu}{L_2 \xi}.$ We can choose such $n$, that 
            \[
                \|\tilde z_n - z^*\|_2 \le R_n \le \frac{\mu}{L_2 \xi}.
            \]
            Therefore, the number of restarts is
            \[
                \frac{R}{2^{n - 1}} \le \frac{\mu}{L_2 \xi} \Leftrightarrow n = \left\lceil \log \frac{L_2 R \xi}{\mu} + 1 \right\rceil.
            \]
            Next we switch to Algorithm \ref{alg:crn-spp} and we need to obtain its number of iterations until convergence. 
            % For ease of notation we denote points, generated by this algorithm, without tilde $z_i$. 
            Denote by $\e'$ the accuracy of solution in terms of the merit function \eqref{eq:merit function}. Owing to first order Lipschitzness of $F(z)$ and the fact that $F(z^*) = 0$, we can get
            \begin{equation}\label{eq:tilde_e_residual}
                \e' = m(z_k) = \frac{1}{2} \| F(z_k) \|_2^2 = \frac{1}{2} \|F(z_k) - F(z^*) \|_2^2 \le \frac{L_1^2}{2} \|z_k - z^*\|_2^2.
            \end{equation}
            Now we establish a connection between the solution in terms of merit function $m(z)$ and the duality gap $G(x, y)$. From \eqref{eq:tilde_e_residual} and \eqref{eq:merit function and dual gap} we get the following:
            \begin{gather}
                \e_G = G(x, y) = \max_{y' \in \mathbb R^n} f(x, y') - \min_{x' \in \mathbb R^n} f(x', y) \le \frac{L_1}{\mu^2} m(z_k) = \frac{L_1}{\mu^2} \e' \notag \\ 
                \Leftrightarrow \frac{\mu^2 \e_G}{L_1} \le \e'. \label{eq:tilde_eps_and_eps}
            \end{gather}
            Then, from \eqref{eq:crn-spp quadratic conv}, \eqref{eq:crn-spp quadratic conv area}, \eqref{eq:tilde_e_residual} and \eqref{eq:tilde_eps_and_eps} we can obtain the needed number of iterations $k$
            \begin{gather*}
                \frac{\mu^2 \e_G}{L_1} \overset{\eqref{eq:tilde_e_residual}, \eqref{eq:tilde_eps_and_eps}}{\le} \frac{L_1^2}{2} \|z_k - z^*\|_2^2 \\
                \overset{\eqref{eq:crn-spp quadratic conv}}{\le} \frac{L_1^2}{2} \bigg( \frac{L_1 L_2}{\mu^2} \|z_{k - 1} - z^*\|_2^2 \bigg)^2 \le \frac{L_1^2}{2} \bigg( \frac{L_1 L_2}{\mu^2}
                \bigg( \frac{L_1 L_2}{\mu^2} \|z_{k - 2} - z^*\|_2^2 \bigg)^2 \bigg)^2 \le ... \\ 
                \le \frac{L_1^2}{2} \bigg( \frac{L_1 L_2}{\mu^2} \bigg)^{2^{k - 1} - 2} \|z_1 - z^*\|_2^{2^{k}} \overset{\eqref{eq:crn-spp quadratic conv area}}{\le} \frac{L_1^2}{2} \bigg( \frac{L_1 L_2}{\mu^2} \bigg)^{2^{k - 1} - 2} \bigg( \frac{\mu^2}{L_1 L_2} \bigg)^{2^k}  \\
                \Leftrightarrow \frac{2 \mu^2 \e_G}{L_1^3} \le \bigg( \frac{\mu^2}{L_1 L_2} \bigg)^{2^{k-1} + 2} \Leftrightarrow \log \frac{2 \mu^2 \e_G}{L_1^3} \le (2^{k - 1} + 2) \log \frac{\mu^2}{L_1 L_2}
            \end{gather*}
            Since $\log (\mu^2 / L_1 L_2) < 0$,
            \begin{equation*}
                \log \frac{2 \mu^2 \e_G}{L_1^3} \le 2^{k - 1} \log \frac{\mu^2}{L_1 L_2} \Leftrightarrow k = \Bigg\lceil \log \frac{\log \frac{L_1^3}{2 \mu^2 \e_G}}{\log \frac{L_1 L_2}{\mu^2}} \Bigg\rceil + 1.
            \end{equation*}
            
            Finally, the total number of iterations of Algorithm \ref{alg:restarted highordermirrorprox with crn-spp} is
            \begin{align*}
                N &= \sum_{i = 1}^n T_i + k \\
                &\le \left( \frac{64 L_p R^{p - 1}}{\mu} \right)^\frac{2}{p + 1} \left\lceil \log \frac{L_2 \xi R}{\mu} + 1 \right\rceil + \left\lceil \log \frac{L_2 \xi R}{\mu} + 1 \right\rceil + \left\lceil \log \frac{\log \frac{L_1^3}{2 \mu^2 \e_G}}{\log \frac{L_1 L_2}{\mu^2}} \right\rceil + 1 \\
                &= O \left( \left( \frac{L_p R^{p - 1}}{\mu} \right)^\frac{2}{p + 1} \log \frac{L_2 \xi R}{\mu} + \log \frac{\log \frac{L_1^3}{2 \mu^2 \e_G}}{\log \frac{L_1 L_2}{\mu^2}} \right)
            \end{align*}
            \qed
        \end{proof}
        
    \subsection{Gradient norm minimization}\label{sec:gradient norm}
    
        % In this subsection we analyze complexity of the Algorithm \ref{alg:restarted highordermirrorprox with crn-spp} in terms of the norm of the gradient $\|\nabla g(x, y)\|_2$.
        
        In this subsection we apply the framework from \cite{dvurechensky2019near} to Algorithm \ref{alg:restarted highordermirrorprox with crn-spp}, introduce Algorithm \ref{alg:restarted highordermirrorprox with crn-spp gradient norm} for problem \eqref{eq:saddle problem gradient norm approximate} and analyze its complexity in terms of the norm of the gradient $\|\nabla g(x, y)\|_2$.
        
        Firstly, we need to introduce some technical lemmas.
        \begin{lemma}
            If $g(x, y)$ is $p$-Lipchitz \eqref{eq:p-lipschitz}, then its partial $p$-th order derivatives are also Lipschitz.
            \begin{equation}\label{eq:objective partial p-lipschitz}
            \forall \hat x, x \in \R^n,\ \hat y, y \in \R^m \Rightarrow \|\nabla_{x^i y^{p - i}}^p g(\hat x, \hat y) - \nabla_{x^i y^{p - i}}^p g(x, y)\|_2 \le L_p \|\hat z - z\|_2.
            \end{equation}
        \end{lemma}
        \begin{proof}
            Here we provide proof only for $\nabla_{x...x}^p$. For other partial derivatives the proof is analogous.
            
            From definition of $\|\cdot\|_2$  \
            \begin{align*}
                \|\nabla_{x...x}^p g(\hat x, \hat y) - \nabla_{x...x}^p g(x, y)\|_2 &= \max_{\|s\|_2 \le 1} |(\nabla_{x...x}^p g(\hat x, \hat y) - \nabla_{x...x}^p g(x, y))[s]^p| \\
                &= \max_{\|s\|_2 \le 1} \Bigg| (\nabla^p g(\hat x, \hat y) - \nabla^p g(x, y))\bigg[ \begin{pmatrix} s \\ 0 \end{pmatrix} \bigg]^p \Bigg| \\
                &\le \max_{\|h\|_2 \le 1} |(\nabla^p g(\hat x, \hat y) - \nabla^p g(x, y))[h]^p| \\
                &= \|\nabla^p g(\hat x, \hat y) - \nabla^p g(x, y)\|_2 \le L_p \|\hat z - z\|_2.
            \end{align*}
            \qed
        \end{proof}
        
        \begin{lemma}\label{lemma:corollary from p-lipschitzness}
            Let $\nabla^p_{x...x} g(x, y)$ be Lipschitz \eqref{eq:objective partial p-lipschitz}. Then
            \begin{equation}\label{eq:corollary from p-lipschitzness}
                \forall n \in [p] \Rightarrow \|\nabla_{x...x}^{p - n} g(\hat z) - \nabla_{x...x}^{p - n} \Phi_{(x, y), p}(\hat z) \|_2 \le \frac{L_p (\sqrt{2})^{n}}{(n + 1)!} \|\hat z - z\|_2^{n + 1}.
            \end{equation}
        \end{lemma}
        \begin{proof}
            We prove this by induction. 
            
            The base of induction $n = 1$ follows from the definition of Taylor approximation. Denote $f(z) = \nabla_{x...x}^{p-1} g(z)$.
            \begin{gather*}
                \|\nabla^{p - 1}_{x...x} g(\hat z) - \nabla^{p - 1}_{x...x} \Phi_{(x, y), p}(\hat z)\|_2 \\
                =\|\nabla^{p - 1}_{x...x} g(\hat z) - \nabla^{p - 1}_{x...x} g(z) - \nabla^{p}_{x...xx} g(z)[\hat x - x] - \nabla^{p}_{x...xy} g(z)[\hat y - y]\|_2 \\
                = \|f(\hat z) - f(z) - \nabla f(z)[\hat z - z]\|_2 \\
                = \|\int_0^1 \la \nabla f(z + \tau(\hat z - z)) - \nabla f(z); \hat z - z \ra d \tau \|_2 \\
                \le \int_0^1 \bigg\| 
                \begin{pmatrix} 
                    \nabla_{x...xx}^p g(z + \tau (\hat z - z)) \\ 
                    \nabla_{x...xy}^p g(z + \tau (\hat z - z)) 
                \end{pmatrix} 
                - 
                \begin{pmatrix} 
                    \nabla_{x...xx}^p g(z) \\ 
                    \nabla_{x...xy}^p g(z)) 
                \end{pmatrix} \bigg\|_2 
                \| \hat z - z \|_2 d \tau \\
                = \int_0^1 \sqrt{\|\nabla^p_{x...xx} g(z + \tau (\hat z - z)) - \nabla^p_{x...xx} g(z)\|_2^2 + \|\nabla^p_{x...xy} g(z + \tau (\hat z - z)) - \nabla^p_{x...xy} g(z)\|_2^2} \cdot \\
                \cdot \|\hat z - z\|_2 d \tau \\ 
                \overset{\eqref{eq:objective partial p-lipschitz}}{\le} \sqrt 2 L_p \|\hat z - z\|_2^2 \int_0^1 \tau d \tau = \frac{L_p \sqrt 2}{2} \|\hat z - z\|_2^2.
            \end{gather*}
            
            Now assume it holds for $n = p - 1$:
            \begin{gather}
                \|\nabla_x g(\hat z) - \nabla_x \Phi_{(x, y), p}(\hat z) \|_2 \notag \\
                = \bigg\| \nabla_x g(\hat z) - \nabla_x g(z) - (\nabla_{xx}^2 g(z)[\hat x - x] - \nabla_{xy}^2 g(z)[\hat y - y])- ... - \notag \\
                - \nabla_x \bigg(\frac{1}{p!} \nabla^p g(z)[\hat z - z]^p \bigg) \bigg\|_2 \notag \\
                \le \frac{L_p (\sqrt{2})^{p - 1}}{p!} \|\hat z - z\|_2^{p}. \label{eq:n = p - 1}
            \end{gather}
            And consider $n = p$
            \begin{gather*}
                |g(\hat z) - \Phi_{(x, y), p}(\hat z)| \\
                = |g(\hat z) - g(z) - \nabla_x g(z)[\hat x - x] - \nabla_y g(z)[\hat y - y] - ... -  \frac{1}{p!}\nabla^p g(z) [\hat z - z]^p| \\
                \le \int_0^1
                \bigg\|
                \begin{pmatrix}
                    \nabla_x g(z + \tau(\hat z - z)) \\
                    \nabla_y g(z + \tau(\hat z - z))
                \end{pmatrix}
                - 
                \begin{pmatrix}
                    \nabla_x g(z) \\
                    \nabla_y g(z)
                \end{pmatrix}
                - \\
                - \tau
                \begin{pmatrix}
                    \nabla^2_{xx} g(z)[\hat x - x] + \nabla^2_{xy} g(z)[\hat y - y] \\
                    \nabla^2_{yx} g(z)[\hat x - x] + \nabla^2_{yy} g(z)[\hat y - y]
                \end{pmatrix}
                - ... - \\
                - \frac{\tau^{p - 1}}{p!}
                \begin{pmatrix}
                    \nabla_x (\nabla^p g(z) [\hat z - z]^p) \\
                    \nabla_y (\nabla^p g(z) [\hat z - z]^p)
                \end{pmatrix}
                \bigg\|_2 \|\hat z - z\|_2 d \tau \\
                = \int_0^1 \Big(\|\nabla_x g(z + \tau(\hat z - z)) - \nabla_x g(z) - \\
                - \tau(\nabla^2_{xx} g(z)[\hat x - x] + \nabla^2_{xy} g(z)[\hat y - y]) - ... - \\
                - \frac{\tau^{p - 1}}{p!} \nabla_x (\nabla^p g(z) [\hat z - z]^p)\|_2^2 + \\
                + \|\nabla_y g(z + \tau(\hat z - z)) - \nabla_y g(z)- \\
                - \tau(\nabla^2_{yx} g(z)[\hat x - x] + \nabla^2_{yy} g(z)[\hat y - y]) - ... - \\ 
                - \frac{\tau^{p - 1}}{p!} \nabla_y (\nabla^p g(z) [\hat z - z]^p)\|_2^2 \Big)^{1/2} \|\hat z - z\|_2 d \tau.
            \end{gather*}
            If we denote $\hat z = z + \tau(\hat z - z)$ in \eqref{eq:n = p - 1}, each of two factors under the square root is indeed what we had for $n = p - 1$. Finally,
            \begin{align*}
                \|\nabla_x g(\hat z) - \nabla_x \Phi_{(x, y), p}(\hat z) \|_2 &\le \sqrt 2 \frac{L_p (\sqrt 2)^{p - 1}}{p!} \|\hat z - z\|_2^{p + 1} \int_0^1 \tau^p d \tau \\
                &= \frac{L_p (\sqrt 2)^{p}}{(p  + 1)!} \|\hat z - z\|_2^{p + 1}.
            \end{align*}
            
            For any other partial derivative in \eqref{eq:corollary from p-lipschitzness} the result is the same and can be obtained in a similar way.
            \qed
        \end{proof}
        
        The next lemma is a modified version of Lemma 5.2 from \cite{grapiglia2019tensor} for SPP.
        
        \begin{lemma}[Lemma 5.2 from \cite{grapiglia2019tensor}]
            Let $(\tilde x, \tilde y) = T_{p,M}^g \left( x, y \right),\ p \ge 2$, where $M \ge \sqrt 2 p L_p > \frac{1}{\sqrt 2} p L_p$ and assumption \ref{ass:p-lipschitz} hold. Then
            \begin{equation}\label{eq:gradient norm and objective residual}
                \|\nabla g(\tilde x, \tilde y)\|_2^\frac{p + 1}{p} \frac{M^\frac{3p + 1}{2p}}{2^\frac{2p^2 + p + 1}{2p} p (p + 1)!} \le g(x, \tilde y) - g(\tilde x, y).
            \end{equation}
        \end{lemma}
        \begin{proof}
            \[
                \|\nabla g(\tilde x, \tilde y)\|_2^2 = \|\nabla_x g(\tilde x, \tilde y)\|_2^2 + \|\nabla_y g(\tilde x, \tilde y)\|_2^2.
            \]
            
            Firstly, consider $\nabla_x$:
            \begin{gather*}
                \|\nabla_x g(\tilde x, \tilde y)\|_2^2 = \|\nabla_x g(\tilde x, \tilde y) - \nabla_x \Phi_{(x, y), p} (\tilde x, \tilde y) + \nabla_x \Phi_{(x, y), p} (\tilde x, \tilde y) - \\
                - \nabla_x \Omega_{(x, y), p, M} (\tilde x, \tilde y) + \nabla_x \Omega_{(x, y), p, M} (\tilde x, \tilde y)\|_2^2\\
                \le \Big( \|\nabla_x g(\tilde x, \tilde y) - \nabla_x \Phi_{(x, y), p} (\tilde x, \tilde y)\|_2 + \\
                + \|\nabla_x \Phi_{(x, y), p} (\tilde x, \tilde y) - \nabla_x \Omega_{(x, y), p, M} (\tilde x, \tilde y)\| + \|\nabla_x \Omega_{(x, y), p, M} (\tilde x, \tilde y)\|_2 \Big)^2 \\
                \le \Bigg( \frac{2^{\frac{p - 1}{2}} L_p}{p!} \|\tilde z - z\|_2^p + \frac{2^{\frac{p - 1}{2}} M}{p!} \|\tilde x - x\|^p_2 \Bigg)^2 \le 2^p M^2 \|\tilde z - z\|_2^{2p}.
            \end{gather*}
            For $\nabla_y$ in a similar way we get the same result
            \begin{equation*}
                \|\nabla_x g(\tilde x, \tilde y)\|_2^2 \le 2^{p} M^2 \|\tilde z - z\|_2^{2p}.
            \end{equation*}
            Summing these two results, we obtain
            \begin{equation}\label{eq:norm grad upper bound}
                \|\nabla g(\tilde x, \tilde y)\|_2^2 \le 2^{p + 1} M \big(\|\tilde x - x\|_2^2 + \|\tilde y - y\|_2^2 \big)^p.
            \end{equation}
            
            Secondly, consider point $(\tilde x, y)$. From \eqref{eq:corollary from p-lipschitzness} it is obvious that
            \[
                |g(\tilde x, y) - \Phi_{(x, y), p}(\tilde x, y)| \le \frac{L_p (\sqrt 2)^p}{(p + 1)!}\|(\tilde x, y) - (x, y)\|_2^{p + 1} = \frac{L_p (\sqrt 2)^p}{(p + 1)!}\|\tilde x - x\|_2^{p + 1}.
            \]
            From this fact we get
            \begin{gather*}
                g(\tilde x, y) \le \Phi_{(x, y), p}(\tilde x, y) + \frac{L_p (\sqrt 2)^{p}}{(p + 1)!} \|\tilde x - x\|_2^{p + 1} \\
                = \Phi_{(x, y), p}(\tilde x, y) + \frac{L_p (\sqrt 2)^{p - 1}}{(p + 1)!} \|\tilde x - x\|_2^{p + 1} - \\
                - \Bigg( \frac{M (\sqrt 2)^{p - 1}}{(p + 1)!} \|\tilde x - x\|_2^{p + 1} - \frac{L_p (\sqrt 2)^p}{(p + 1)!}\|\tilde x - x\|_2^{p + 1} \Bigg) \\
                = \Omega_{(x, y), p, M}(\tilde x, y) - (M - L_p \sqrt 2) \frac{(\sqrt 2)^{p - 1} \|\tilde x - x\|_2^{p + 1}}{(p + 1)!} \\
                \le \Omega_{(x, y), p, M}(\tilde x, \tilde y) - (M - L_p \sqrt 2) \frac{(\sqrt 2)^{p - 1} \|\tilde x - x\|_2^{p + 1}}{(p + 1)!}.
            \end{gather*}
            Since $M \ge \sqrt 2 p L_p \Leftrightarrow - L_p \sqrt 2 \ge -\frac{M}{p}$. we have
            \begin{equation}\label{eq:objective residual by x lower bound}
                \Omega_{(x, y), p, M}(\tilde x, \tilde y) - g(\tilde x, y) \ge \frac{M(p - 1) (\sqrt 2)^{p - 1} \|\tilde x - x\|_2^{p + 1}}{p(p + 1)!} \ge \frac{M \|\tilde x - x\|_2^{p + 1}}{p (p + 1)!}.
            \end{equation}
            
            Now consider the point $(x, \tilde y)$. In a similar way we can get the following result:
            \begin{equation}\label{eq:objective residual by y lower bound}
                g(x, \tilde y) - \Omega_{(x, y), p, M}(\tilde x, \tilde y) \ge \frac{M \|\tilde y - y\|_2^{p + 1}}{p (p + 1)!}.
            \end{equation}
            
            From the sum of \eqref{eq:objective residual by x lower bound} and \eqref{eq:objective residual by y lower bound} we obtain 
            \begin{equation}\label{eq:approximation dual gap lower bound}
                g(x, \tilde y) - g(\tilde x, y) \ge \frac{M}{p (p + 1)!}\Big(\|\tilde x - x\|_2^{p + 1} + \|\tilde y - y\|_2^{p + 1} \Big).
            \end{equation}
            
            Finally, we need to connect \eqref{eq:norm grad upper bound} and \eqref{eq:approximation dual gap lower bound}. From H\"older's inequality we can get 
            \[
                \left( \sum_{i=1}^n x_i^p \right)^{\frac{1}{p}} \le n^{\frac{q - p}{qp}} \left( \sum_{i = 1}^n x_i^q \right)^\frac{1}{q},
            \]
            where $q, p \in \mathbb N,\ q > p \ge 1$. Now, from \eqref{eq:norm grad upper bound} it follows that
            \[
                \left( \frac{\|\nabla g(\tilde x, \tilde y)\|_2^2}{2^{p + 1}M} \right)^\frac{1}{2p} \le \left(\|\tilde x - x\|_2^2 + \|\tilde y - y\|_2^2 \right)^\frac{1}{2}.
            \]
            And, from \eqref{eq:approximation dual gap lower bound} we can get 
            \[
                \left( \frac{p (p + 1)! (g(x, \tilde y) - g(\tilde x, y))}{M} \right)^\frac{1}{p + 1} \ge \left( \|\tilde x - x\|_2^{p + 1} + \|\tilde y - y\|_2^{p + 1} \right)^\frac{1}{p + 1}.
            \]
            Since $p \ge 2$, we obtain the final result
            \[
                \|\nabla g(\tilde x, \tilde y)\|_2^\frac{p + 1}{p} \frac{M^\frac{3p + 1}{2p}}{2^\frac{2p^2 + p + 1}{2p} p (p + 1)!} \le g(x, \tilde y) - g(\tilde x, y).
            \]
            \qed
        \end{proof}
        
        Now we have all the needed information to estimate the final convergence rate of the Algorithm \ref{alg:restarted highordermirrorprox with crn-spp gradient norm} for gradient norm minimization.
        
        \begin{algorithm}[t]
        	\caption{Restarted HighOrderMirrorProx with local quadratic convergence for gradient norm minimization}\label{alg:restarted highordermirrorprox with crn-spp gradient norm}
        	\begin{algorithmic}[1]
        		\STATE \textbf{Input} $z_1 \in \mathcal Z, p \ge 1, 0 < \e_\nabla < 1, R : R\ge \|z_1 - z^*\|_2$, $\rho \in (0, 1)$, $\alpha \in (0, 1)$.
        		
        		\STATE \textbf{Define}:
        		\begin{gather*}
        		\tilde z_1 = z_1, \quad M = \sqrt 2 p L_p, \quad \mu = \frac{\e}{4 R}, \quad \xi = \max \left\{ 1, \frac{4 R L_1}{\e_\nabla} \right\}, \\
        		\e' = \frac{M^\frac{3p + 1}{2p} \e_\nabla^\frac{p + 1}{p}}{2^\frac{2p^2 + 3p + 3}{2p} p (p + 1)!}, \\
        		g_\mu(x, y) = g(x, y) + \frac{\mu}{2}\big( \|x - x_1\|_2^2 - \|y - y_1\|_2^2 \big).
        		\end{gather*}
        		
                \FOR{$i \in [n]$, where $n = \left\lceil \log \frac{L_2 R \xi}{\mu} + 1 \right\rceil$}
                
                    \STATE Set $R_i = \frac{R}{2^{i - 1}}$
                
        		    \STATE Set $T_i = \left\lceil \bigg( \frac{64L_p R_i^{p - 1}}{p! \mu} \bigg)^\frac{2}{p + 1} \right\rceil$
        		    
        		    \STATE Run Algorithm \ref{alg:highordermirrorprox} for $g_\mu$ with $\tilde z_{i}$, $p$,  $T_i$ as input
        		    
        		    \STATE $\tilde z_{i + 1} = \bar z_{T_i}$
                
                \ENDFOR
                
            	\STATE Run Algorithm \ref{alg:crn-spp} with $\tilde z_{i + 1}$, $\e'$, $\bar \gamma = \frac{L_2 \mu^2}{2 L_1^2}$, $\rho$, $\alpha$, $g_\mu$ as input

        		\STATE \textbf{Find} $\tilde z = T_{p,\ M}^{g_\mu}(z_k)$

        		\STATE \textbf{Output} $\tilde z$.
        	\end{algorithmic}
        \end{algorithm}
        
        \begin{theorem}\label{th:gradient norm}
        	Assume the function $g(x, y): \mathbb R^n \times \mathbb R^m \to \R$ is convex by $x$ and concave by $y$, $p$ times differentiable on $\mathbb R^n$ with $L_p$-Lipschitz $p$-th derivative. Let $\tilde z$ be generated by Algorithm \ref{alg:restarted highordermirrorprox with crn-spp gradient norm}. Then
        	\[
        	    \|\nabla g(\tilde z)\|_2 \le \e_\nabla,
        	\]
        	and the total complexity of Algorithm~\ref{alg:restarted highordermirrorprox with crn-spp gradient norm} is 
            \[
                O \left( \left( \frac{L_p R^p}{\e_\nabla} \right)^\frac{2}{p + 1} \log \frac{L_2 R^2 \xi}{\e_\nabla} \right),
            \]
            where $\xi = \max \left\{ 1, \frac{4 R L_1}{\e} \right\}$.
        \end{theorem}
        \begin{proof}
            Denote $z^*_\mu = (x^*_\mu, y^*_\mu)$ the saddle point of $g_\mu(z)$. First of all, since $g_\mu(x, y)$ is strongly-convex-strongly-concave function, we can apply restart technique to it every time the distance to its saddle point $\|z - z^*_\mu\|_2$ reduces twice. To check this, we consider upper estimate of the distance to the solution of regular function $R: R \ge \|z^* - z\|_2$ and show, that on each $i$-th restart $\|z^*_\mu - z_i\|_2 \le \|z^* - z_i\|_2 \le R_i$. We prove this by induction. 
            \begin{align*}
                g(x^*_\mu, y_1) + \frac{\mu}{2} \|x^*_\mu - x_1\|_2^2 = g_\mu(x^*_\mu, y_1) \le g_\mu(x^*, y_1) &= g(x^*, y_1) + \frac{\mu}{2} \|x^* - x_1\|_2^2 \\
                &\le g(x^*_\mu, y_1) + \frac{\mu}{2} \|x^* - x_1\|_2^2
            \end{align*}
            \[
                \Leftrightarrow \|x^*_\mu - x_1\|_2 \le \|x^* - x_1\|_2.
            \]
            \begin{align*}
                g(x_1, y^*_\mu) - \frac{\mu}{2} \|y^*_\mu - y_1\|_2 = g_\mu(x_1, y^*_\mu) \ge g_\mu(x_1, y^*) &= g(x_1, y^*) - \frac{\mu}{2} \|y^* - y_1\|_2^2 \\
                &\ge g(x_1, y^*_\mu) - \frac{\mu}{2} \|y^* - y_1\|_2^2
            \end{align*}
            \[
                \Leftrightarrow \|y^*_\mu - y_1\|_2 \le \|y^* - y_1\|_2.
            \]
            This gives us
            \[
                \|z^*_\mu - z_1\|_2 \le \|z^* - z_1\|_2 \le R.
            \]
            
            Now suppose, that $\|z^*_\mu - z_i\|_2 \le \|z^* - z_i\|_2 \le R_i = R / 2^{i - 1}$. Consider $i + 1$. From the proof of Theorem \ref{th:restarted highordermirrorprox conv} and our choice of $T_i$ in Algorithm \ref{alg:restarted highordermirrorprox with crn-spp gradient norm}, we know, that
            \begin{gather*}
                \mu \|z_{i + 1} - z^*_\mu\|_2^2 = \mu \|\bar z_{T_i} - z^*_\mu\|_2^2 \le \frac{16 L_p}{p!} \left( \frac{R_i^2}{2 T_i} \right)^\frac{p + 1}{2} \le \mu R_{i + 1}^2 \\
                \Leftrightarrow \|z_{i + 1} - z^*_\mu\|_2 \le R_{i + 1}.
            \end{gather*}
            
            From Theorem \ref{th:restarted highordermirrorprox with local quadratic conv} we already know the number of restarts to reach the area of quadratic convergence: $n = \left\lceil \log \frac{L_2 R \xi}{\mu} + 1 \right\rceil$.
            
            % As we already know from Theorem \ref{th:restarted highordermirrorprox with local quadratic conv}, the number of restarts $n$ of Algorithm \ref{alg:restarted highordermirrorprox} to reach the area of local quadratic convergence of Algorithm \ref{alg:crn-spp} is $n = \left\lceil \log \frac{L_2 R \xi}{\mu} + 1 \right\rceil$.
            
            % Then, we need to find the number of restarts $n$ of Algorithm \ref{alg:restarted highordermirrorprox} to reach the area of local quadratic convergence of Algorithm \ref{alg:crn-spp}: $\|\tilde z_n - z^*\|_2 \le \frac{\mu}{L_2 \xi}.$ We can choose such $n$, that 
            % \[
            %     \|\tilde z_n - z^*\|_2 \le R_n \le \frac{\mu}{L_2 \xi},
            % \]
            % where $\xi = \max \left\{1, \frac{L_1}{\mu} \right\} = \max \left\{ 1, \frac{4 R L_1}{\e} \right\}$. Therefore, the number of restarts is 
            % \[
            %     \frac{R}{2^{n - 1}} \le \frac{\mu}{L_2 \xi} \Leftrightarrow n = \left\lceil \log \frac{L_2 R \xi}{\mu} + 1 \right\rceil.
            % \]
            
            Next, we need to show, that Algorithm \ref{alg:restarted highordermirrorprox with crn-spp gradient norm} converges in terms of $\|\nabla g_\mu(z)\|_2$.
            Let $\tilde z = (\tilde x, \tilde y)$ be the output of Algorithm \ref{alg:restarted highordermirrorprox with crn-spp gradient norm}. From the definition of $g_\mu$ we get
            \begin{align*}
                \|\nabla g(\tilde x, \tilde y)\|_2^2 &= \|\nabla_x g_\mu(\tilde x, \tilde y) - \mu(\tilde x - x_1)\|_2^2 + \|\nabla_y g_\mu(\tilde x, \tilde y) + \mu(\tilde y - y_1)\|_2^2 \\
                &\le \left( \|\nabla_x g_\mu(\tilde x, \tilde y)\|_2 + \mu\|\tilde x - x\|_2 \right)^2 + \left( \|\nabla_y g_\mu(\tilde x, \tilde y)\|_2 + \mu\|\tilde y - y\|_2 \right)^2 \\
                &\le 2 \left( \|\nabla_x g_\mu(\tilde x, \tilde y)\|_2^2 + \|\nabla_y g_\mu(\tilde x, \tilde y)\|_2^2 \right) + 2 \mu^2 \left( \|\tilde x - x\|_2^2 + \|\tilde y - y\|_2^2 \right) \\
                &= 2\|\nabla g_\mu(\tilde x, \tilde y)\|_2^2 + 2\mu^2\|\tilde z - z_1\|_2^2
            \end{align*}
            \[
                \Leftrightarrow \|\nabla g(\tilde x, \tilde y)\|_2 \le \sqrt{2\|\nabla g_\mu(\tilde x, \tilde y)\|_2^2 + 2\mu^2\|\tilde z - z_1\|_2^2}.
            \]
            
            Firstly, we estimate $\|\nabla g_\mu(\tilde x, \tilde y)\|_2$. From \eqref{eq:gradient norm and objective residual} we know, that
            \begin{align*}
                \|\nabla g_\mu(\tilde x, \tilde y)\|_2^\frac{p + 1}{p} \frac{M^\frac{3p + 1}{2p}}{2^\frac{2p^2 + p + 1}{2p} p (p + 1)!} &\overset{\eqref{eq:gradient norm and objective residual}}{\le} g_\mu(x, \tilde y) - g_\mu(\tilde x, y) \\
                &\le \max_{\tilde y \in \R^m} g_\mu(x, \tilde y) - \min_{\tilde x \in \R^n} g_\mu(\tilde x, y) = G_\mu(x, y) \le \e'.
            \end{align*}
            \begin{equation}\label{eq:gradient norm upper bounded e/2}
                \Leftrightarrow \|\nabla g_\mu(\tilde x, \tilde y)\|_2 \le \left( \frac{2^\frac{2p^2 + p + 1}{2p} p (p + 1)! \e'}{M^\frac{3p + 1}{2p}} \right)^\frac{p}{p + 1} = \frac{\e_\nabla}{2}.
            \end{equation}
            
            Secondly, we estimate $\mu\|\tilde z - z_1\|_2$. By definition of $R$ we know, that
            \[
                \|z^* - z_1\|_2 \le R.
            \]
            And since $\tilde z$ is closer to solution then $z_1$, we have
            \[
                \|\tilde z - z^*\|_2 \le \|z^* - z_1\|2 \le R.
            \]
            From these facts and triangle inequality we get
            \begin{equation}\label{eq:initial distance to tensor step point upper bounded e/2}
                \mu \|\tilde z - z_1\|_2 \le \mu \left( \|\tilde z - z^*\|_2 + \|z^* - z_1\|_2 \right) \le 2 R \mu = \frac{\e_\nabla}{2}.
            \end{equation}
            
            Thus, from \eqref{eq:gradient norm upper bounded e/2} and \eqref{eq:initial distance to tensor step point upper bounded e/2} we obtain
            \[
                \|\nabla g_\mu(\tilde x, \tilde y)\|_2 \le \sqrt{2\e_\nabla^2 / 4 + 2 \e_\nabla^2 / 4} = \e_\nabla.
            \]
            
            Finally, we need to estimate complexity of the Algorithm \ref{alg:restarted highordermirrorprox with crn-spp gradient norm}.
            \begin{align*}
                N = \sum_{i = 1}^n T_i + k &\le \left( \frac{64 L_p}{p! \mu}  \right)^\frac{2}{p + 1} \sum_{i = 1}^n R_i^\frac{2(p - 1)}{p + 1} + n + k \\
                &\le  \left( \frac{64 L_p R^{p - 1}}{p! \mu} \right)^\frac{2}{p + 1} \cdot n + n + k \\
                &= O \left( \left( \frac{L_p R^p}{\e_\nabla} \right)^\frac{2}{p + 1} \log \frac{L_2 R^2 \xi}{\e_\nabla} \right),
            \end{align*}
            where $\xi = \max \left\{ 1, \frac{4 R L_1}{\e_\nabla} \right\}$.
            Here $k$ is the number of iterations of Algorithm \ref{alg:crn-spp} inside Algorithm \ref{alg:restarted highordermirrorprox with crn-spp gradient norm}. We dropped it due to its $\log \log$ dependence on $\e_\nabla$.
            
            \qed
        \end{proof}

\section{Discussion}\label{sec:discussion}

    In this work we propose three methods for $p$-th order tensor methods for strongly-convex-strongly-concave SPP. Two of these methods tackle classical minimax SPP \eqref{eq:saddle problem} and MVI \eqref{eq:mvi} problems, and the third method aims at gradient norm minimization of SPP \eqref{eq:saddle problem gradient norm}. 
    
    The methods for minimax problem are based on the ideas, developed in the works \cite{bullins2020higher} and \cite{huang2020cubic}. In \cite{bullins2020higher} the authors use $p$-th order oracle to construct an algorithm for MVI problems with monotone operator. As a corollary, this algorithm allows to solve SPP with convex-concave objective. Because of strong convexity and strong concavity of our problem, we can apply a restart technique to the method from \cite{bullins2020higher} and get better algorithm complexity. To further improve local convergence rate we switch to the algorithm from \cite{huang2020cubic} in the area of its quadratic convergence. This way we get rid of the multiplicative logarithmic factor and get additive $\log\log$ factor in the final complexity estimate and get locally quadratic convergence.
    
    The method for gradient norm minimization relies on the works \cite{grapiglia2019tensor} and \cite{dvurechensky2019near}. From \cite{grapiglia2019tensor} we take the result, that connects norm of the gradient of the objective with objective residual, and slightly modify it for SPP. This step allows us to use the framework from \cite{dvurechensky2019near} and use our optimal algorithm for minimax SPP for gradient norm minimization.
    
    In spite of all the improvements, we should remind about many additional assumptions about the problem, which reduces number of real problems, that can suit to it. 
    
    One of possible directions for further research are the more general H\"older conditios instead of Lipschitz conditions and uniformly convex case. Additionally, the author in \cite{bullins2020higher} provided implementation details of the Algorithm \ref{alg:highordermirrorprox} only for $p = 2$. Therefore, the questions about its realizaition for $p > 2$ are still opened.

% \begin{acknowledgements}
%     The work of P.~Ostroukhov was fulfilled in Sirius (Sochi) in August 2020.  The research  of  P.~Dvurechensky was partially supported by the Ministry of Science and Higher Education of the Russian Federation (Goszadaniye) 075-00337-20-03, project no. 0714-2020-0005. The research of A.~Gasnikov was partially  supported by RFBR, project number 19-31-51001. The work of P.~Ostroukhov was supported by Andrei M. Raigorodskii Scholarship in Optimization.
% \end{acknowledgements}


\bibliographystyle{spmpsci}      % mathematics and physical sciences
\bibliography{bibliography.bib}   % name your BibTeX data base

\end{document}